\documentclass[12pt,leqno]{article}
\usepackage{latexsym}
\usepackage{amsfonts,amsmath,amssymb}
\usepackage{enumitem}
\usepackage{amsthm}
\usepackage{graphics, epsfig}
\usepackage{xcolor}
\usepackage{marvosym}
\usepackage{mathtools}
\usepackage{esint}

\newtheorem{remark}{Remark}[section]

\newtheorem{thm}{Theorem}[section]
\newtheorem*{prth1.2}{Proof of Theorem 1.2}
\newtheorem*{prth1.3}{Proof of Theorem 1.3}
\newtheorem*{prcor1.2}{Proof of Corollary 1.2}
\numberwithin{equation}{section}
\numberwithin{theorem}{section}
\newtheorem{lem}[thm]{Lemma}
\numberwithin{proposition}{section}
\numberwithin{lemma}{section}
\numberwithin{corollary}{section}
\numberwithin{remark}{section}

\newcommand{\Tmax}{T_{\rm max}}

\newcommand{\RN}{\mathbb{R}^N}
\newcommand{\R}{\mathbb{R}}

\def\Xint#1{\mathchoice
    {\XXint\displaystyle\textstyle{#1}}%
    {\XXint\textstyle\scriptstyle{#1}}%
    {\XXint\scriptstyle\scriptscriptstyle{#1}}%
    {\XXint\scriptscriptstyle\scriptscriptstyle{#1}}%
    \!\int}
\def\XXint#1#2#3{\setbox0=\hbox{$#1{#2#3}{\int}$}
    \vcenter{\hbox{$#2#3$}}\kern-0.5\wd0}

\def\dashint{\Xint{\raise4pt\hbox to7pt{\hrulefill}}}
\begin{document}
\thispagestyle{empty}
\setcounter{page}{1}
\noindent
\begin{center}
{\bf { \Large Blow-up phenomena for a  chemotaxis system with flux limitation
}}

\begin{center}
    \vspace*{1cm} 
{\bf 
{ M.Marras \footnote{ Dipartimento di Matematica e Informatica, Universit\'a di Cagliari, via Ospedale 72, 09124 Cagliari (Italy), mmarras@unica.it},
S.Vernier-Piro \footnote{ Facolt\'a di Ingegneria e Architettura, Universit\'a di Cagliari, Viale Merello 92, 09123 Cagliari (Italy), svernier@unica.it},
 T.Yokota \footnote{Department of Mathematics, Tokyo University of Science, 1-3, Kagurazaka, Shinjuku-ku, Tokyo 162-8601 (Japan), yokota@rs.tus.ac.jp} }}\\
\end{center}

\end{center}


\begin{abstract}
In this paper we consider nonnegative solutions of the following parabolic-elliptic cross-diffusion system

\begin{equation*}
\left\{ \begin{array}{l}
\begin{aligned}
&u_t = \Delta u  -  \nabla(u  f(|\nabla v|^2 )\nabla v),   \\[6pt]
&0= \Delta v -\mu +  u ,  \quad \int_{\Omega}v =0,  \ \ \mu := \frac 1 {|\Omega|} \int_{\Omega} u  dx, \\[6pt]
&u(x,0)= u_0(x),
\end{aligned}
\end{array} \right.
\end{equation*}
in $\Omega \times (0,\infty)$, with $\Omega$ a ball in $\RN$, $N\geq 3$ under homogeneous Neumann boundary conditions and $f(\xi) = (1+ \xi)^{-\alpha}$, $0<\alpha < \frac{N-2}{2(N-1)}$, which describes gradient-dependent limitation of cross diffusion fluxes.
Under conditions on $f$ and initial data, we prove that a solution which blows up in finite time in $L^\infty$-norm, blows up also in $L^p$-norm for some $p>1$.  Moreover, a lower bound of blow-up time is derived. 

\end{abstract}


\section{Introduction} \label{Intro}

In this paper we consider the chemotaxis system with flux limitation,
\begin{equation} \label{sys1}
\begin{cases}
u_t = \Delta u  -  \nabla(u  f(|\nabla v|^2 )\nabla v),   \quad & x\in \Omega, \ t>0,\\[2mm]
0= \Delta v -\mu +  u ,  \qquad \qquad \qquad &x \in \Omega, \ t>0, \\[2mm]
\frac{\partial u}{\partial \nu}=\frac{\partial v}{\partial \nu}= 0, \quad  \qquad \qquad \qquad &x \in \partial \Omega, \ t>0, \\[2mm]
u(x,0)= u_0(x), \qquad  \qquad &x \in \Omega,
 \end{cases}
\end{equation}
with $\Omega$ a ball in $\mathbb{R}^N$, $N\geq 3$, $\mu=  \frac 1{ |\Omega|}\int u dx >0$, $\int_{\Omega} v dx=0$, $f \in C^2([0,\infty))$. 
We assume that the initial data $u_0(x) \in C^0(\overline\Omega)$, $u_0 \geq 0$.\\
System \eqref{sys1} is a modified version of the well known Keller--Segel model
 \begin{equation*} \label{KL}
\begin{cases}
u_t = \Delta u  -  \nabla(u \nabla v),    & x \in \Omega , \ t>0,\\[6pt]
v_t= \Delta v -\mu +  u ,   & x \in \Omega, \ t>0, 
\end{cases}
\end{equation*}
proposed by Keller and Segel \cite{KS_1970} in 1970, which is a mathematical model describing aggregation phenomena of organisms due to chemotaxis, i.e., the directed movement of cell density $u(x,t)$ at the position $x$ and at the time $t$ in response to the gradient of a chemical attractant $v(x,t)$. The presence of the elliptic equation in \eqref{sys1} instead of the parabolic one reflects the situation where the chemicals diffuse much faster than cells move.\\
 For decades various Keller--Segel type systems have been extensively studied by many authors.\\
 In \cite{BBTW}, the authors propose a very exhaustive survey and analysis focused on classical and modified Keller--Segel models. Moreover, other contributions (e.g., \cite{IY}, \cite{MTV}, \cite{MVP}, \cite{MVV} and \cite{MS}) investigate the behavior of the solutions to chemotaxis systems, specifically boundedness, decay, blow-up properties and non-degeneracy of blow-up points. For more general Keller--Segel systems involving three equations of fully parabolic type or parabolic-elliptic-elliptic type,  see \cite{CT}, \cite{CMTY} and the reference therein. \\
 The finite-time blow-up of nonradial solutions of \eqref{sys1} is investigate in \cite{Nagai}, where some conditions on the mass and the moment of the initial data are introduced, with $f(|\nabla v|^2)=1$ on a bounded domain in $\mathbb{R}^2$.\\
Bellomo and Winkler \cite{BW} consider the following chemotaxis system
\begin{equation} \label{sysBW}
\left\{ \begin{array}{l}
\begin{aligned}
&u_t = \nabla \cdot \Big( \frac{ u \nabla u}{\sqrt{u^2+|\nabla u|^2}} \Big)  - \chi \nabla \cdot \Big( \frac{ u \nabla v} {\sqrt{1+|\nabla v|^2}} \Big),   \\[6pt]
&0= \Delta v -\mu +  u ,
\end{aligned}
\end{array} \right.
\end{equation}
under the initial condition $u_0(x)>0$ and no-flux boundary conditions, when the spatial domain $\Omega$ is a ball in $\mathbb{R}^N$, $N\geq 1$, $ \mu := \frac 1 {|\Omega|} \int_{\Omega} u_0  dx$. The authors prove that if $\chi>1$ then, for any choice of $m$ with
\begin{equation*}
\begin{cases}
m> \frac 1 {\sqrt{\chi^2 -1}} , & \text{if} \ N=1,   \\[2mm]
m>0 \ \ {\rm is\ arbitary,} & \text{if} \ N\geq 2,
\end{cases}
\end{equation*}
there exist positive initial data $u_0 \in C^3(\overline{\Omega})$, $\int_{\Omega} u_0 dx=m$,  which are such that the problem \eqref{sysBW} possesses,  for some $T_{max}>0$, a uniquely determined classical solution $(u,v)$ in $\Omega \times (0,T_{max})$, blowing up at time $T_{max}$ in the sense that $\displaystyle \limsup_{t \nearrow T_{max}} \|u(x,t)\|_{L^{\infty}} = \infty$. These results are a continuation of the analytical study presents in \cite{BW2} of the flux-limited chemotaxis model \eqref{sysBW} in which the main results assert the existence of a unique classical solution of \eqref{sysBW}, extensible in time up to a maximal $T_{max} \in (0,\infty]$ which has the property that if $T_{max} <\infty$ then  $\displaystyle \limsup_{t \nearrow T_{max}} \|u(\cdot,t)\|_{L^{\infty}} = \infty$.\\
In \cite{ChMT} Chiyoda et al.\ consider the system 
\begin{equation} \label{sysChMT}
\left\{ \begin{array}{l}
\begin{aligned}
&u_t = \nabla \cdot \Big( \frac{ u^p \nabla u}{\sqrt{u^2+|\nabla u|^2}} \Big)  - \chi \nabla \cdot \Big( \frac{u^q \nabla v}{\sqrt{1+|\nabla v|^2}}  \Big),   \\[6pt]
&0= \Delta v -\mu +  u ,
\end{aligned}
\end{array} \right.
\end{equation}
in a ball in $\mathbb{R}^N$, $N\in \mathbb{N}$, under no-flux boundary conditions and initial condition $u_0(x)>0$.  
Assuming suitable conditions for $\chi$ and $u_0$  when  $1 \leq p \leq q$, they obtain existence of blow-up solutions  of \eqref{sysChMT}. When $p=q=1$ the system \eqref{sysChMT} reduces to \eqref{sysBW}.\\
In \cite{MOT} Mizukami et al.\ for the solutions to the problem \eqref{sysChMT} obtain 
\begin{itemize}
\item  if $p,q  \ge 1,$ local existence and extensibility criterion ruling out gradient blow-up; 
\item if $p> q+1 - \frac 1 N$, global existence and boundedness. 
\end{itemize}

Negreanu and Tello in \cite{NT} consider the case when $f(|\nabla v|^2)= \chi  |\nabla v|^{p-2}$, i.e.,
\begin{equation}\label{Negreanu}
\begin{cases}
u_t = \Delta u  -   \nabla \cdot (\chi u  |\nabla v|^{p-2}\nabla v),   \\[2mm]
0= \Delta v -\mu +  u, 
\end{cases}
\end{equation}
with homogeneous Neumann boundary conditions and nonnegative initial data $u_0(x)$ with $ \mu := \frac 1 {|\Omega|} \int_{\Omega} u_0  dx$, $\chi$ a positive constant and $p$ so that
\begin{equation*}
\begin{cases}
p\in (1, \infty),  & \text{if} \ \ N=1, \\[2mm]
p \in \Big(1, \frac{N}{N-1} \Big), & \text{if} \ \ N\geq 2.
\end{cases}
\end{equation*}

Under suitable assumptions on the data, they obtain for the solutions of \eqref{Negreanu} uniform bounds in $ L^{\infty}(\Omega) $ and the global existence, while for the one-dimensional case, the existence of infinitely many non-constant steady-states for $p\in(1,2)$ for any $\chi$ positive and a given positive mass is obtained.\\
In this paper we focus our attention on blow-up phenomena, extensively studied both in the elliptic and in the parabolic cases (see for instance \cite{MP}, \cite{MV} and references therein).\\
For the solutions of \eqref{sys1}, in \cite{W2020}, Winkler proves that, if $f(\xi) \geq (1+\xi)^{-\alpha}$ with $0<\alpha < \frac{N-2}{2(N-1)}$, then throughout a considerably large set of radially symmetric initial data, the blow-up phenomenon, with respect to the $ L^{\infty}$ norm of $u$, occurs in finite time.\\
This result is contained in the following theorem.

\begin{thm}[\cite{W2020} Finite-time blow-up in $L^\infty$-norm]\label{BULinfty}
Let\/ $\Omega\equiv B_R(0) \subset \R^N$, $N \geq 3$ and 
$R>0$, and let $f$ satisfy
\begin{equation}\label{f}
f\in C^2([0, \infty)), \  \text{as well as} \ f(\xi) \ge k_f (1+\xi)^{-\alpha} \quad \text{for all} \ \xi\geq 0
\end{equation}
with some $k_f >0$ and 
\begin{equation}\label{alpha}
0< \alpha < \frac{N-2}{2(N-1)}.
\end{equation}
Then for any choice of $\mu>0$ one can find $R_0=R_0(\mu) \in (0,R)$ with the property that whenever $u_0$ satisfies
\begin{equation}\label{u_0}
 u_0\in C^0(\overline{\Omega}), \ u_0 \ \text{nonnegative with} \ \frac 1 {|\Omega|}\int_{\Omega} u_0 dx =\mu >0
\end{equation}
 and
 \begin{equation} \label{int u_0}
u_0 \ is \ radially \ \text{symmetric with} \ \int_{B_{r}(0)} u_0 dx \geq \int_{\Omega} u_0 dx, \ \forall \ r\in(0,R)
\end{equation}
as well as 
\begin{equation}\label{int BR_0}
\frac 1 {|\Omega|}\int_{B_{R_0}(0)}u_0 dx \geq \frac {\mu} 2 \Big(\frac R{R_0}\Big)^N,
\end{equation}
the corresponding solution $(u,v)$ of \eqref{sys1} blows up in finite time; that is, for the uniquely determined local classical solution, maximally extended up to some time $T_{max} \in (0,\infty]$ according to Lemma \ref{LSE} below, we then have $T_{max} <\infty$ and 
\begin{align}\label{blowupinfty}
\limsup_{t \nearrow T_{max}} \|u(\cdot, t)\|_{L^\infty(\Omega)}=\infty.
\end{align}
\end{thm}

\vspace{4mm}

The first purpose of this paper is to prove that the solutions of \eqref{sys1} blow up in $L^p$-norm, for some $p>1$, if they blow up in $L^{\infty}$-norm.\\  

\begin{thm}[Finite-time blow-up in $L^p$-norm]\label{BULp}
Let\/ $\Omega \equiv B_R(0) \subset\mathbb{R}^N$, $N\geq 3$ and $R>0$. 
Then, a classical solution $(u, v)$ of \eqref{sys1}  for 
$t \in (0, T_{max})$ and with $f(\xi)= k_f (1+\xi)^{-\alpha}$ with some $k_f >0$,
provided by Theorem~\ref{BULinfty}, is such that for all $\frac{N}{2}<p< N$, 
\begin{align*}
\limsup_{t \nearrow T_{max}}\, 
\left\|u(\cdot,t)\right\|_{L^{p}(\Omega)} 
= \infty.
\end{align*}
\end{thm}

\vspace{4mm}

The second purpose of this paper is to study the behavior of the solutions of \eqref{sys1} near the blow-up time $T_{max}$.\\
Since it is not always possible to compute $T_{max}$, deriving a lower bound is a matter of great importance, in order to obtain a safe time interval of existence of the solution $(0, T)$ with $T<T_{max}$.\\
With this aim, we define for all $p>1$ the auxiliary function
\begin{equation}\label{Psi} 
\Psi(t):= \frac 1 {p} \|u(\cdot,t)\|^{p}_{L^{p}(\Omega)} \quad {\rm with}\quad  \Psi_0 := \Psi(0)= \frac 1 {p} \|u_0\|^{p}_{L^{p}(\Omega)}.
\end{equation}

\begin{thm}[Lower bound of blow-up time]\label{LB}

Let\/ $\Omega \equiv B_R(0) \subset\mathbb{R}^N$, $N\geq 3$, $R>0$ and let $\Psi$ be defined in \eqref{Psi}. 
Then, for all  $\frac N 2 <p<N$ and some positive constants $B_1, B_2, B_3, B_4$, 
the blow-up time $T_{max}$ for \eqref{sys1} with $f(\xi)= k_f(1+\xi)^{-\alpha}$ 
with some $k_f >0$, provided by Theorem~\ref{BULinfty}, 
satisfies the estimate
\begin{align}\label{lower Tmax in Lp}
T_{max}\geq T:= \int_{\Psi_0}^{\infty}\frac{d\eta}{B_1 \eta + B_2\eta^{\gamma_1} + B_3 \eta^{\gamma_2} +B_4 \eta^{\gamma_3}},
  \end{align}
with $\gamma_1:= \frac{p+1}{p}$, 
$\gamma_2:=\frac{2(p+1)-N} {2p-N}, \ \ \gamma_3:=\frac{2(p+1) - \frac{N(p+1)(1+\epsilon)}{p+1+\epsilon} }{2p-\frac{N(1+\epsilon)(p+1)}{p+1+\epsilon}} .$
\end{thm}

\vspace{4mm}

The scheme of this paper is the following: Section \ref{prel} is concerned with preliminaries including the Neumann heat semigroup, in Section \ref{blow up in L^p}, since the solution of \eqref{sys1} blows up in finite time in $L^{\infty}$-norm we prove that the solution  blows up also in $L^p$-norm (for some $p>1$). Section \ref{lower bound} is devoted to find appropriate assumptions on the data, such that the $\|u\|_{L^p(\Omega)}$ remains bounded in $(0,T)$ with $T< T_{max}$. Clearly this value of $T$ provides a lower bound for blow-up time $T_{max}$ of $u$.


\section{Preliminaries} \label{prel}
In this section, we present some preliminary lemmata which we shall use in the proof of our main results.

\begin{lem}[see \cite{W2020}] \label{LSE} 
Let $\Omega  \subset  \RN$, $N \ge1$  be a bounded domain with smooth boundary, and assume that $f$ and $u_0$ satisfy  \eqref{f} and \eqref{u_0}.
Then there exists $T_{max} \in(0, \infty]$ an a uniquely determined pair $(u,v)$ of functions 
\begin{align*}
    &u\in C^0(\overline{\Omega}\times[0,T_{max}))
           \cap C^{2,1}(\overline{\Omega}\times(0,T_{max})),\\[1.05mm]
    &v\in \bigcap_{q>N}L^{\infty}_{loc} ([0, T_{max}); W^{1,q}(\Omega))
           \cap C^{2,0}(\overline{\Omega}\times(0,T_{max})),
\end{align*}
with $u\geq 0$ and $v \geq 0$, in $\Omega \times (0, T_{max})$, such that $(u,v)$ solves \eqref{sys1} classically in $\Omega \times (0, T_{max})$, with
\begin{equation} \label{int u=int u0}
\int_{\Omega} u(\cdot, t)= \int_{\Omega} u_0 \quad for \ all \ t\in (0, T_{max})
\end{equation}
and
    \begin{align*} 
            {\it if}\ T_{max}<\infty,
    \quad 
            {\it then}\ \limsup_{t \nearrow T_{max}} 
                           \|u(\cdot,t)\|_{L^\infty(\Omega)}=\infty.
    \end{align*}
    Moreover, if $\Omega = B_{R}(0)$ with some $R>0$ and $u_0$ is radially simmetric with respect to $x=0$, then also $u(\cdot, t)$ and $v(\cdot, t)$ are radially symmetric for each $t\in (0,T_{max})$.
%
\end{lem}

We next give some properties of the Neumann heat semigroup 
which will be used later. 
For the proof, see \cite[Lemma 2.1]{Cao} and \cite[Lemma~1.3]{W-2010}.

\begin{lem} \label{Cao} 
Let $(e^{t \Delta})_{t\geq 0}$ be the Neumann heat semigroup in 
$\Omega$, and let $\mu_1 >0$ denote the first non zero eigenvalue of 
$-\Delta$ in $\Omega$ under Neumann boundary conditions. 
Then there exist $k_1,  k_2 >0$ which 
 depend only on $\Omega$ and have  
the following properties\/{\rm :}
\begin{enumerate}[label=(\roman*)]
\item if\/ $1 \leq q\leq {\rm p}\leq \infty$, then 
\begin{equation} \label{etDeltaz}
\|e^{t \Delta} z\|_{L^{{\rm p}}(\Omega)} \leq k_1t^{ - \frac N 2(\frac 1 q - \frac 1 {\rm p})} 
\|z\|_{L^q(\Omega)}, \ \ \forall \ t >0
\end{equation}
holds for all $z\in L^q(\Omega)$.
\item If\/ $1< q \leq {\rm p} \leq \infty$, then
\begin{equation} \label{etDelta nablaz}
\|e^{t \Delta} \nabla \cdot \textbf{z\,}\|_{L^{{\rm p}}(\Omega)} \leq k_2\big(1+ t^{-\frac 1 2 - \frac N 2(\frac 1 q - \frac 1 {\rm p})}\big) e^{-\mu_1 t} \|\textbf{z\,}\|_{L^q(\Omega)}, \ \ \forall \ t >0
\end{equation}
is valid for any $\textbf{z\,} \in (L^{q}(\Omega))^N$, 
where $e^{t \Delta} \nabla \cdot {}$ is the extension of the operator
$e^{t \Delta} \nabla \cdot {}$ on $(C_0^\infty(\Omega))^N$ to $(L^q(\Omega))^N$. 
\end{enumerate}
\end{lem}

 In Section \ref{lower bound} we will use the  Gagliardo--Nirenberg inequality in the following form. 

\begin{lem}\label{lemma GN ineq}
Let\/ $\Omega$ be a  bounded and smooth domain of\/ $\mathbb{R}^N$ with $N\geq 1$. 
Let $ \mathsf{r}\geq 1$, $0< \mathsf{q} \leq  \mathsf{p}\leq\infty$, $ \mathsf{s}>0$.
	Then there exists a constant
	$C_{{\rm GN}}>0$ such that 
\begin{equation} \label{GN ineq}	
		\|f\|^{\mathsf{p}}_{L^{ \mathsf{p}}(\Omega)}\leq C_{{\rm GN}} \Big(\|\nabla f\|^{\mathsf{p} a}_{L^{ \mathsf{r}}(\Omega)}
		\|f\|_{L^{ \mathsf{q}}(\Omega)}^{{\mathsf{p}}(1-a)}
	+\|f\|^{\mathsf{p}}_{L^{ \mathsf{s}}(\Omega)}\Big)
	\end{equation}
for all $f\in L^{\textsf{q}}({\Omega})$ with 
$\nabla f \in (L^{\textsf{r}}(\Omega))^N$
and
$a:=\frac{\frac{1}{ \mathsf{q}}-\frac{1}{ \mathsf{p}}}
		{\frac{1}{ \mathsf{q}}+\frac{1}{N}-
		\frac{1}{ \mathsf{r}}} \in [0,1]$. 	
\begin{proof}
Following from the Gagliardo--Nirenberg inequality 
(see \cite{Nir} for more details): 
\begin{equation*}\label{GNfirst} 	
		\|f\|^{ \mathsf{p}}_{L^{ \mathsf{p}}(\Omega)}\leq \Big[c_{{\rm GN}} \Big(\|\nabla f\|^{a}_{L^{ \mathsf{r}}(\Omega)}
		\|f\|_{L^{ \mathsf{q}}(\Omega)}^{1-a} +\|f\|_{L^{ \mathsf{s}}(\Omega)}\Big) \Big]^{ \mathsf{p}},
	\end{equation*}
with some $c_{{\rm GN}}>0$, and then from the inequality
\begin{equation*}
(\mathsf{a}+\mathsf{b})^{\mathsf{p}} \leq 2^{\mathsf{p}}(\mathsf{a}^{\mathsf{p}} + 
\mathsf{b}^{\mathsf{p}})\quad {\rm for\ any}\ \mathsf{a}, \mathsf{b}\geq 0,
\ \mathsf{p}>0,
\end{equation*}
we arrive to \eqref{GN ineq} with $C_{\rm GN}= 2^{\mathsf{p}} c_{\rm GN}^{\mathsf{p}}$.
\end{proof}
\end{lem}


\section{Blow-up in $L^{p}$-norm}\label{blow up in L^p}
The aim of this section is to prove Theorem \ref{BULp}. To this end, first we prove the following lemma.
\begin{lem}\label{LemmaBoundedness u nabla v} 
Let\/ $\Omega \subset\mathbb{R}^N,\ N\geq 3$ be a bounded and smooth domain. 
Let $(u,v)$ be a classical solution of system \eqref{sys1} with $f(\xi)= k_f(1+\xi)^{-\alpha}$ with some $k_f >0$. 
If $\alpha$ satisfies \eqref{alpha} and if for some 
$\frac{N}{2} <p<N $
there exists 
$C>0$
such that 
\begin{align*}
\left\|u(\cdot,t)\right\|_{
L^{p}(\Omega)} 
\leq C, \quad \textrm{ for any } t \in (0,T_{max}),
\end{align*}
then, for some $\hat C>0$, 
\begin{align}\label{BoundednessU^infty}
\left\|u(\cdot,t)\right\|_{L^{\infty}(\Omega)} 
\leq {\hat C}, \quad \textrm{ for any } t \in (0,T_{max}).
\end{align}
\end{lem}

\begin{proof}
For any $t\in (0,T_{max})$, we set $t_0 := \max \{0, t-1\}$ and we consider the representation formula for $u$:
\begin{align*}
& u(\cdot ,t) = e^{(t-t_0)\Delta} u( \cdot, t_0) - k_f\int_{t_0}^{t} e^{(t-s)\Delta} \nabla \cdot \Big(u (\cdot, s)\frac{\nabla v(\cdot,s)}{(1+ |\nabla v(\cdot , t)|^2)^{\alpha}}\Big)\,ds\notag\\
&=: u_1(\cdot,t)+u_2(\cdot,t)
\end{align*}
and
\begin{align} \label{unorm}
 \| u(\cdot ,t) \|_{L^{\infty}} \leq \| u_1(\cdot,t) \|_{L^{\infty}(\Omega)} +\|u_2(\cdot,t)\|_{L^{\infty}(\Omega)}.
\end{align}
We have
\begin{equation}\label{u_1}
\begin{split}                                                                                                                                                                                                                                                                                                                                                                                                                                                                                                                                                                                                                                                                                                                                                                                                                                                                                                                                                                                                                                                                                                                                                                                                                                                                                                                                                                                                                                                                                                                                                                                                                                                                                                                                                                                                                                                                                                                                                                                                                                                                                                                                                                                                                                                                                                                                                                                                                                                                                                                                                                                                                                                                                                                                                                                                                                                                                                                                                                                                                                                                     \lVert u_1 (\cdot,t) \rVert _{L^{\infty}(\Omega)} \leq \max \{\lVert u_0 \rVert_{L^{\infty}(\Omega)}, \mu |\Omega|k_1\} =:C_1,
\end{split}
\end{equation}
with $k_1>0$ and $\mu$ defined in \eqref{u_0}. 
In fact, if $t\leq 1$, then $t_0=0$ and hence the maximum principle yields $u_1(\cdot, t) \leq \| u_0\|_{L^{\infty}(\Omega)}$. 
If $t>1$, then $t-t_0=1$ and from \eqref{int u=int u0} and \eqref{etDeltaz} with ${\rm p}=\infty$ and $q=1$, we deduce that $\lVert u_1(\cdot,t)\rVert_{L^{\infty}(\Omega)} \leq k_1 (t-t_0)^{-\frac N 2} \lVert u(\cdot,t_0) \rVert_{L^1(\Omega)} \leq \mu |\Omega|k_1$.\\
\vskip.2cm
We next use \eqref{etDelta nablaz} with ${\rm p}=\infty$, which leads to
\begin{align}\label{u_2} 
& \|u_2 (\cdot, t) \|_{L^{\infty}(\Omega)} \\
\notag & \leq k_2 k_f  \int_{t_0}^{t} ( 1 + (t-s)^{-\frac 1 2 - \frac{N}{2q} }) e^{-\mu_1  (t-s)} \|u (\cdot, s)\frac{ \nabla v(\cdot,s)}{(1+|\nabla v|^2)^{\alpha}}  \|_{L^q(\Omega)} \,ds \\
\notag &\leq k \int_{t_0}^{t} ( 1 + (t-s)^{-\frac 1 2 - \frac{N}{2q} }) e^{-\mu_1  (t-s)} \|u (\cdot, s) |\nabla v|^{1- 2\alpha}  \|_{L^q(\Omega)} \,ds,
\end{align}
with $k:=k_2 k_f$ and $\frac {|\nabla v|}{(1+|\nabla v|^2)^{\alpha}}\leq |\nabla v|^{1-2\alpha}$.\\

Here, we may assume that 
$\frac N 2<p<N$,
and then 
we can fix $N< q < \frac{N p}{N-p}=p^*$. Since $2 \alpha < 1$, by H$\ddot{{\rm o}}$lder's inequality, we can estimate the last term in \eqref{u_2} as
\begin{align*}
& \|u (\cdot, s)|\nabla v(\cdot,s)|^{1-2\alpha}  \|_{L^q(\Omega)}\\
&\le 
\|u (\cdot, s)\|_{L^{\frac{q}{2\alpha}}(\Omega)}\|\nabla v(\cdot,s)  \|^{1-2\alpha}_{L^q(\Omega)}\\
&\le C_2\|u (\cdot, s)\|_{L^{\frac q {2\alpha}}(\Omega)} \|\nabla v(\cdot,s)\|^{1-2\alpha}_{L^{p^*}(\Omega)}\quad 
{\rm for\ all}\ s \in (0, \Tmax),
\end{align*}
for some $C_2>0$. 
The Sobolev embedding theorem and elliptic regularity theory 
applied to the second equation in \eqref{sys1} tell us that 
$\|v(\cdot,s)\|_{W^{1,p^*}(\Omega)}
\leq C_3\|v(\cdot,s)\|_{W^{2,p}(\Omega)}
\le C_4$ with some $C_3, C_4>0$. 
Thus again by H$\ddot{{\rm o}}$lder's inequality, the definition of $\mu = \frac 1 {|\Omega|}\int_{\Omega} u_0 dx $ and interpolation's inequality, we obtain
\begin{align*}
&\|u (\cdot, s)|\nabla v(\cdot,s)|^{1-2\alpha}  \|_{L^q(\Omega)} \\ 
&\le 
C_5\|u (\cdot, s)\|_{L^{\frac q{2\alpha}}(\Omega)}\\
&\le C_{5}\|u (\cdot, s)\|^{\theta}_{L^{\infty}(\Omega)} \|u (\cdot, s)\|^{1-\theta}_{L^1(\Omega)}\\
&\le C_{6}\|u (\cdot, s)\|^{\theta}_{L^{\infty}(\Omega)}\quad 
{\rm for\ all}\ s \in (0, \Tmax),
\end{align*}
with $\theta := 1 - \frac{2\alpha}{q} \in (0,1)$, $C_5:=C_2C_4$ and $C_{6}:=C_5(\mu |\Omega|)^{1-\theta}$. 
Hence, combining this estimate and \eqref{u_2}, we infer
\begin{align*}
\|u_2 (\cdot, t) \|_{L^{\infty}(\Omega)}
\leq C_{6}k_2  \int_{t_0}^{t} ( 1 +  (t-s)^{-\frac 1 2 - \frac{N}{2q} }) e^{-\mu_1  (t-s)} \|u (\cdot, s)\|_{L^{\infty}(\Omega)}^{ \theta}\,ds.
\end{align*}
Now fix any $T \in (0, T_{max})$. 
Then, since $t-t_0\leq 1$, we have
\begin{align}\label{u_21}
\|u_2 (\cdot, t) \|_{L^{\infty}(\Omega)}
&\leq C_{6}k_2   \int_{t_0}^{t} ( 1 +  (t-s)^{-\frac 1 2 - \frac{N}{2q} } e^{-\mu_1  (t-s)}) \,ds \cdot \sup_{t \in [0, T]} \|u (\cdot, t)\|_{L^{\infty}(\Omega)}^{ \theta}\notag\\[6pt]
&\leq C_{7}\sup_{t \in [0, T]} \|u (\cdot, t)\|_{L^{\infty}(\Omega)}^{ \theta},
\end{align}

where $C_{7}:=C_{6}k_2(1+\mu_1^{\frac{N}{2q}-\frac{1}{2}}\int_0^\infty r^{-\frac 1 2 - \frac{N}{2q}} e^{-r}\,dr)>0$ is finite, 
because $\frac{1}{2}+\frac{N}{2q}<1$  (i.e., $q>N$). \\
\vskip.2cm
Plugging \eqref{u_1} and \eqref{u_21},
into \eqref{unorm},	we see that
\begin{align*}
\| u(\cdot, t) \|_{L^{\infty}} \le C_1+C_{7} \sup_{t \in [0, T]} \|u (\cdot, t)\|_{L^{\infty}(\Omega)}^{ \theta},
\end{align*}
which implies
\begin{align*}
\sup_{t \in [0, T]}\|u(\cdot, t)\|_{L^\infty(\Omega)} 
&\le C_{1}+C_{7} \Big(\sup_{t \in [0, T]} \|u (\cdot, t)\|_{L^{\infty}(\Omega)}\Big)^{\theta}\quad {\rm for\ all}\ T \in (0, T_{max}).
\end{align*} 
From this inequality with $\theta \in (0,1)$, we arrive at \eqref{BoundednessU^infty}.
\end{proof}

\begin{prth1.2}
\rm{Since Theorem~\ref{BULinfty} holds, the unique local classical solution of \eqref{sys1} blows up at $t=T_{max}$ in the sense of \eqref{blowupinfty}, that is,  
$\limsup_{t \nearrow \Tmax}\| u(\cdot , t) \|_{L^{\infty}(\Omega)}= \infty$.
We prove that it blows up also in $L^p$-norm by contradiction.\\ 
In fact, if one supposes that there exist $p> \frac {N}{2} $ and $C>0$ 
such that 
\begin{align*}
\|u(\cdot, t)\|_{L^{p}(\Omega)} \leq C,\quad 
{\rm for\ all}\ t \in (0, T_{max}),
\end{align*} 
then, from Lemma \ref{LemmaBoundedness u nabla v}, it would exist 
$\hat C>0$ such that 
\begin{equation*} 
\lVert u(\cdot,t)\rVert _{L^\infty(\Omega)}\leq \hat C,\quad
{\rm for\ all}\ t \in (0, T_{max}),
\end{equation*}
which contradics \eqref{blowupinfty}. Thus, 
if $u$ blows up in $L^{\infty}$-norm, 
then $u$ blows up also in $L^p$-norm 
for all $p>\frac{N}{2}$}. \qed
\end{prth1.2}


\section{Lower bound of the blow-up time $\Tmax$} \label{lower bound}
Throughout this section we assume that Theorem \ref{BULp} holds.\\

We want to obtain a safe interval of existence of the solution of \eqref{sys1}  $[0,T]$, with $T$ a lower bound of the blow-up time $T_{max}$. To this end, first we construct a first order differential inequality for $\Psi$ defined in \eqref{Psi} and by integration we get the lower bound.

\begin{prth1.3} 
\rm{By differentiating  \eqref{Psi} we have
\begin{align}
 \label{Psi'}
&\Psi'(t)= \int_\Omega u^{p -1} \Delta u \,dx
- \int_\Omega u^{p -1}\nabla\cdot (u \nabla v f(|\nabla v|^2 ) \, dx\\
& \notag=:\mathcal  I_1+  \mathcal I_2,
\end{align}
with
\begin{align}
 \label{I1}
&\mathcal  I_1 = \int_\Omega u^{p -1} \Delta u \,dx\\ 
\notag&=  \int_\Omega \nabla \cdot\big( u^{p -1}\nabla u\big) dx - (p-1) \int_{\Omega}  u^{p-2} | \nabla u|^2 dx\\ 
\notag&  =  - \frac{4(p-1)}{p^2} \int_{\Omega}  | \nabla u^{\frac p 2}|^2 dx.
\end{align}
In the second term of \eqref{Psi'}, integrating by parts and using the boundary conditions in \eqref{sys1}, $\forall \ t\in [0,T_{max})$  we obtain
\begin{align}\label{I2}
&\mathcal  I_2 =- \int_{\Omega} u^{p -1}\nabla\cdot (u \nabla v f(|\nabla v|^2 ) \, dx \\
\notag&= (p-1) \int_{\Omega}  f(|\nabla v|^2)    u^{p -1} \nabla u \nabla vdx\\
\notag&=\frac{p-1}{p}  \int_{\Omega} \nabla u^p \cdot \nabla v f (|\nabla v|^2) dx \\
\notag&= -\frac{p-1}{p} \int_{\Omega} u^p \nabla \cdot[ \nabla v f (|\nabla v|^2)] dx\\
\notag&= -\frac{p-1}{p} \int_{\Omega} u^p [\Delta v  f (|\nabla v|^2)] dx\\
\notag&- \frac{p-1}{p} \int_{\Omega} u^p f'(|\nabla v|^2) \nabla v \cdot \nabla (|\nabla v|^2) dx.
\end{align}
Using the second equation of \eqref{sys1} and taking into account that $f(\xi)= k_f(1+\xi)^{-\alpha}$, $  f'(\xi )=  -\alpha k_f(1+\xi)^{-\alpha-1}$ in \eqref{I2}, we have
\begin{align}
 \label{I2 bis}
&\mathcal I_2  = - k_f\frac{p-1}{p} \int_{\Omega} u^p \frac{\mu - u}{(1+ |\nabla v|^2)^{\alpha}}   dx\\
 \notag& + \alpha k_f \frac{p-1}{p} \int_{\Omega}u^p \frac{ \nabla v \cdot \nabla (|\nabla v|^2) }{(1 + |\nabla v|^2)^{\alpha + 1}}dx\\ 
\notag & \leq k_f\frac{p-1}{p} \int_{\Omega}  u^{p +1}  dx  +  \alpha k_f  \frac{p-1}{p} \int_{\Omega}  u^{p} \frac{\nabla v \cdot \nabla (|\nabla v|^2)}{(1+ |\nabla v|^2)^{\alpha +1}} dx,
\end{align}
where we dropped the negative term $- k_f\frac{p-1}{p} \int_{\Omega} u^p \frac{\mu }{(1+ |\nabla v|^2)^{\alpha}}   dx$ and used the inequality $\frac{1}{(1+ |\nabla v|^2)^{\alpha} }\leq 1$ as $\alpha>0$. \\
In order to estimate the second term of \eqref{I2 bis} we recall the radially symmetric setting to obtain
\begin{align*}
&\int_{\Omega}  u^{p} \frac{\nabla v \cdot \nabla (|\nabla v|^2)}{(1+ |\nabla v|^2)^{\alpha +1}} dx= \omega_N \int_0^R u^p \frac{Nv_r(v^2_r)_r}{(1+ v^2_r)^{\alpha +1}} r^{N-1} dr\\
& =2N\omega_N \int_0^R u^p \frac{v^2_r v_{rr}}{(1+ v^2_r)^{\alpha +1}} r^{N-1} dr,
\end{align*}
which together with $v_{rr}= \frac{\mu}{N} - u + \frac{N-1} {r^N} \int_0^r \rho^{N-1} u \ d  \rho $ implies
\begin{align}
 \label{I2 4}
&\int_{\Omega}  u^{p} \frac{\nabla v \cdot \nabla (|\nabla v|^2)}{(1+ |\nabla v|^2)^{\alpha +1}} dx\\
\notag&= 2\mu \omega_N \int_0^R u^p \frac{v^2_r }{(1+ v^2_r)^{\alpha +1}} r^{N-1} dr\\
\notag& - 2N\omega_N \int_0^R u^{p+1}  \frac{v^2_r}{(1+ v^2_r)^{\alpha +1}} r^{N-1} dr \\
\notag &  + 2N(N-1) \omega_N \int_0^R u^p \frac{v^2_r }{(1+ v^2_r)^{\alpha +1}}  \frac 1 r \Big(\int_0^r \rho^{N-1} u d \rho \Big) dr \\
\notag& \leq 2\mu \omega_N \int_0^R u^p r^{N-1} dr  + 2N(N-1) \omega_N \int_0^R u^p  \frac 1 r \Big(\int_0^r \rho^{N-1} u d \rho \Big) dr,
\end{align}
where we dropped the negative term $- 2N\omega_N \int_0^R u^{p+1}  \frac{v^2_r}{(1+ v^2_r)^{\alpha +1}} r^{N-1} dr$ and used the inequality $\frac{v^2_r}{(1+ v^2_r)^{\alpha + 1} }\leq 1.$ \\
In the second term of \eqref {I2 4}, H$\ddot{{\rm o}}$lder's inequality yelds that for all $\epsilon >0$ there exists $c= c(\epsilon, N, p)$ such that
\begin{align}
 \label{I2 5}
& \omega_N \int_0^R u^p  \frac 1 r  \Big(\int_0^r \rho^{N-1} u d \rho \Big) dr \\
 \notag& \leq  \omega_N \int_0^R u^p  \frac 1 r  \Big(\int_0^r \rho^{N-1} d \rho \Big)^{\frac {p}{p+1}}\Big( \int_0^r u^{p+1} \rho^{N-1} d \rho\Big)^{\frac 1{p+1}} dr\\
 \notag& \leq \Big(\frac 1 {N}\Big)^{\frac p {p+1}} \Big(\int_{\Omega} u^{p+1} dx \Big)^{\frac 1 {p+1}} \omega^{\frac p {p+1}}_N \int_0^R u^p r^{\frac {Np}{p+1}-1} dr \\
 \notag& \leq \Big( \! \frac {1} N \! \Big)^{\frac p{p+1}} \! \Big( \! \int_{\Omega} u^{p+1} dx \! \Big)^{ \frac 1 {p+1}} \omega_N^{\frac p {p+1}}\! \Big(\! \int_0^R u^{p+1+\epsilon} r^{N-1} dr\!  \Big)^{\frac p {p+1+\epsilon}} \Big(\! \int_0^R  r^{\frac{\epsilon Np}{ p+1} - 1} dr\!  \Big)^{\frac{1+\epsilon}{p+1+\epsilon}}\\
\notag& = c \Big( \int_{\Omega} u^{p+1} dx\!  \Big)^{\frac 1 {p+1} } \Big(\int_{\Omega} u^{p+1+\epsilon} dx\!\Big)^{\frac{p}{p+1+ \epsilon}}.
\end{align}
Combining \eqref{I2 5} and \eqref{I2 4} with \eqref{I2 bis} we obtain
\begin{align}
 \label{I2 6}
 & \mathcal I_2  \leq 2 \alpha \mu k_f\frac{p-1}{p}   \int_{\Omega} u^p dx  + k_f \frac{p-1}{p} \int_{\Omega}  u^{p +1}  dx \\
 \notag&+ 2 \alpha N(N-1) c k_f \;  \frac{p-1}{p} \Big( \int_{\Omega} u^{p+1} dx \Big)^{\frac 1 {p+1}}\Big( \int_{\Omega} u^{p+1+\epsilon} dx\Big)^{\frac {p}{p+1+\epsilon}} \\
 \notag& \leq c_1  \int_{\Omega} u^p dx  + c_2  \int_{\Omega} u^{p +1}  dx + c_3\Big( \int_{\Omega} u^{p+1+\epsilon} dx\Big)^{\frac {p+1}{p+1+\epsilon}}
 \end{align}
where, in the last term, we used Young's inequality 
 with $c_1 = 2 \alpha \mu k_f\frac{p-1}{p}  , \ \  c_2 =  k_f\frac{p-1}{p}+ 2 \alpha N(N-1) c k_f\,  \frac{p-1}{p(p+1)}, \ \ c_3= 2 \alpha N(N-1) c  k_f\;  \frac{p-1}{p+1} $. \\
Thanks to the Gagliardo--Nirenberg inequality \eqref{GN ineq}, with $\mathsf{p}= 2\frac{p+1}{p}, \   \mathsf{r} = \mathsf{q}= \mathsf{s}=2, \ a= \theta_0 := \frac {N}{2(p+1)} \in(0,1)$ for all $p> \frac N 2$, we see that
\begin{align}
 \label{u^p+1}
 & \int_{\Omega} u^{p+1} dx = \| u^{\frac p 2}\|_{L^{2\frac{p+1}{p}}(\Omega)}^{2\frac{p+1}{p}} \\
&  \notag  \leq C_{GN} \| \nabla u^{\frac p 2} \|_{L^2(\Omega)}^{ 2\frac{p+1}{p}\theta_0}  \| u^{\frac p 2}\|_{L^2(\Omega)}^{ 2\frac{p+1}{p}(1-\theta_0)}+ C_{GN}   \| u^{\frac p 2}\|_{L^2(\Omega)}^{ 2\frac{p+1}{p}}\\
&  \notag = C_{GN}  \Big( \int_{\Omega} |\nabla u^{\frac p 2}|^2 dx\Big) ^{ \frac{N}{2p}}  \Big(  \int_{\Omega} u^{ p} dx\Big) ^{ \frac{2(p+1) - N}{2p}}+  C_{GN}   \Big(\int_{\Omega} u^{p} dx\Big)^{\frac{p+1}p}.
  \end{align}
Applying Young's inequality at the first term of \eqref{u^p+1} we have
\begin{align}
 \label{u^p+1 bis}
 &\int_{\Omega} u^{p+1} dx  \leq  \frac{N}{2p} \epsilon_1 C_{GN}  \int_{\Omega}  |\nabla u^{\frac p 2}|^2  dx\\
 & \notag + C_{GN}  \frac{2p-N}{2p \epsilon_1^{\frac N{2p-N}}} \Big(\int_{\Omega}  u^p dx\Big)^{\frac {2(p+1)-N}{2p-N}} + C_{GN}  \Big(\int_{\Omega} u^{p} dx\Big)^{\frac{p+1}p}
 \end{align}
with $\epsilon_1 >0$ to be choose later on,  and also
\begin{align}\label{u^p+1+ epsilon}
&\Big(\int_{\Omega} u^{p+1+\epsilon} dx \Big)^{\frac {p+1}{p+1+\epsilon}} = \| u^{\frac p 2}\|^{2\frac{p+1}p}_{L^2\frac{p+1+\epsilon}{p}(\Omega)} \\
& \notag \leq  C_{GN} \| \nabla u^{\frac p 2} \|_{L^2(\Omega)}^{ 2\frac{p+1}{p}\theta_{\epsilon}}  \| u^{\frac p 2}\|_{L^2(\Omega)}^{ 2\frac{p+1}{p}(1-\theta_{\epsilon})}+ C_{GN}   \| u^{\frac p 2}\|_{L^2(\Omega)}^{ 2\frac{p+1}{p}}\\
& \notag  =  C_{GN}  \Big(  \int_{\Omega}  |\nabla u^{\frac p 2}|^2  dx \Big)^{\frac{p+1}{p}\theta_{\epsilon}}  \Big(   \int_{\Omega} u^p dx \Big)^{ \frac{p+1 }{p}(1-\theta_{\epsilon})}\\
& \notag  +  C_{GN}  \Big(\int_{\Omega} u^{p} dx\Big)^{\frac{p+1}p},
\end{align}
\vskip.2cm
with $\mathsf{p}= 2\frac {p+1}{p},\  \mathsf{r} = \mathsf{q}= \mathsf{s}=2, \ a= \theta_{\epsilon} := \frac {N(1+ \epsilon)}{2(p+1+\epsilon)} \in(0,1)$ for all $p> \frac N 2$ and sufficiently small $\epsilon >0$.\\
Now, in the first term of \eqref{u^p+1+ epsilon}, we apply the Young's inequality to obtain
\begin{align}
 \label{u^p+1+ epsilon bis}
&  \Big(\int_{\Omega} u^{p+1+\epsilon} dx \Big)^{\frac {p+1}{p+1+\epsilon}}  \\
 \notag  & \leq  c_4 \int_{\Omega}  |\nabla u^{\frac p 2}|^2  dx + c_5\Big(\int_{\Omega}  u^p dx\Big)^{\sigma} +  C_{GN}  \Big(\int_{\Omega} u^{p} dx\Big)^{\frac{p+1}p},
\end{align}
\vskip.2cm
with 
\begin{align*}
&c_4:= \frac{N(1+\epsilon)(p+1) }{2p(p+1 + \epsilon)}  C_{GN} , \quad
c_5:= C_{GN} \Big(\frac{2p (p+1+ \epsilon)- N(p+1) (1+\epsilon)}{2p(p+1+\epsilon)} \Big),\\
&\sigma :=\frac{2(p+1) - \frac{N(p+1)(1+\epsilon)}{p+1+\epsilon} }{2p-\frac{N(1+\epsilon)(p+1)}{p+1+\epsilon}}.
\end{align*}
Note that we can fix $\epsilon>0$ such that $2p -N(1+\epsilon) >0$.\\
Plugging \eqref{u^p+1 bis} and \eqref{u^p+1+ epsilon bis} into \eqref{I2 6} leads to
\begin{align}
 \label{I2 final}
& \mathcal I_2  \leq C \int_{\Omega}  |\nabla u^{\frac p 2}|^2  dx + c_1 \int_{\Omega} u^p dx + C_{GN}  \Big(\int_{\Omega}  u^p dx\Big)^{\frac {p+1}{p}} \\
&  \notag  +  \tilde c_1\Big(\int_{\Omega}  u^p dx\Big)^{\frac {2(p+1)-N}{2p-N}}+c_5 \Big(\int_{\Omega} u^{p} dx\Big)^{\sigma}  
 \end{align}
with $C:=  \frac{N}{2p} \epsilon_1 C_{GN} + c_1, \ \epsilon_1>0, \ \ \tilde c_1:= C_{GN}  \frac{2p-N}{2p \epsilon_1^{\frac N{2p-N}}}c_2.$\\ 
Finally, combining \eqref{I2 final} with \eqref{Psi'} and \eqref{I1} and choosing $\epsilon_1$ such that the term containing $\int_{\Omega} |\nabla u^{\frac p 2}|^2 dx$ vanishes,  we arrive at
\begin{align}\label {Psi' final}
\Psi' \leq c_1 \Psi + C_{GN} \Psi^{\frac {p+1}{p}} + \tilde c_1 \Psi^{\frac {2(p+1)-N}{2p-N}} + c_5 \Psi^{ \sigma}.
\end{align}
Integrating \eqref{Psi' final} from $0$ to $T_{max}$, we arrive at the desired lower bound \eqref{lower Tmax in Lp} with $B_1:= c_1$, $B_2:= C_{GN}$, $B_3:= \tilde c_1$, $B_4:= c_5$ and $\gamma_1 := \frac {p+1}{p}, \ \gamma_2:= \frac {2(p+1)-N}{2p-N}, \ \gamma_3:= \sigma$.}\qed 
\end{prth1.3}
 
\vspace{4mm}

 \remark
\rm{ We note that it is possible to reduce \eqref{Psi' final} so as to have an explicit expression of the lower bound $T$ of $T_{max}$. In fact, since $\Psi(t)$ blows up at time $T_{max}$, there exists a time $t_1 \in (0, T_{max})$ such that $\Psi(t) \geq \Psi_0$ for all $t\in (t_1, T_{max})$.
 Thus, taking into account that 
 \begin{align*}
1<  \gamma_1 <\gamma_2 
\end{align*}
and putting $\gamma := \max \{  \gamma_2, \gamma_3 \}$ 
we have 
 \begin{align} \label{Psi^gamma}
& \Psi  \leq \Psi^{\gamma} \Psi_0^{1-\gamma}, \\
& \notag \Psi^{\gamma_i}  \leq \Psi^{\gamma} \Psi_0^{\gamma_i - \gamma} ,\ \ \ i=1,2,3.
\end{align}
From \eqref{Psi' final} and \eqref{Psi^gamma} we arrive at
\begin{align}\label {Psi' final new}
\Psi' \leq A \Psi^{\gamma}, \ \  \forall t\in (t_1, T_{max}),
\end{align}
with $A:= B_1 \Psi_0^{1-\gamma} + B_2 \Psi_0^{ \gamma_1-\gamma} + B_3   \Psi_0^{\gamma_2 - \gamma} + B_4 \Psi_0^{ \gamma_3-\gamma }$, and $\Psi_0$ in \eqref{Psi}.\\

Integrating \eqref{Psi' final new} from $t=0$ to $t= T_{max}$, we obtain 
\begin{align}\label {lower}
\frac 1 { (\gamma -1) \Psi_0^{\gamma-1}}= \int_{\Psi_0}^{\infty} \frac {d\eta}{\eta^{\gamma}} \leq A \int_{t_1}^{T_{max}} d\tau \leq A \int_{0}^{T_{max}} d \tau =A T_{max}.
\end{align}
We conclude, by \eqref{lower}, that the solution of \eqref{sys1} is bounded in $[0,T]$ with $T:= \frac 1 { A (\gamma -1) \Psi_0^{\gamma-1}}.$

}

\section*{Acknowledgments}
M. Marras and S. Vernier-Piro are members of the Gruppo Nazionale per l'Analisi Matematica, la Probabilit$\grave{\rm a}$ e le loro Applicazioni (GNAMPA) of the Istituto Nazionale di Alta Matematica (INdAM). 

\subsection*{Financial disclosure}

M. Marras is partially supported by the research project: {\it Evolutive and stationary Partial Differential Equations with a focus on biomathematics} (Fondazione di Sardegna 2019), and by the grant PRIN n. PRIN-2017AYM8XW: {\it Non-linear Differential Problems via Variational, Topological and Set-valued Methods}.\\
T. Yokota is partially supported by Grant-in-Aid for Scientific Research (C), No. 21K03278, and by Tokyo University of Science Grant for International Joint Research.

  \end{document}